\documentclass[11pt]{amsart}

\usepackage{amsmath}
\usepackage{amsthm}
\usepackage{amsfonts}
\usepackage{amssymb,latexsym}
\usepackage[all]{xy}
\usepackage{graphicx}
\usepackage[mathscr]{eucal}
\usepackage{verbatim}
\usepackage{hyperref}

\theoremstyle{plain}
    \newtheorem{thm}{Theorem}[section]
    \newtheorem{prop}[thm]{Proposition}
    \newtheorem{lemma}[thm]{Lemma}
    
    \newtheorem{cor}[thm]{Corollary}

\theoremstyle{definition}
    
    \newtheorem{rem}[thm]{Remark}

\theoremstyle{remark}
    \newtheorem{example}[thm]{Example}

\numberwithin{equation}{section}

\newcommand{\rar}{\ensuremath{\rightarrow}}
\newcommand{\lrar}{\ensuremath{\longrightarrow}}

\newcommand{\ints}{\mathbb{Z}}
\newcommand{\ord}{\text{ord}}

\newcommand{\intsn}{\mathbb{Z}_n}
\newcommand{\intsnm}{\mathbb{Z}_n \times \mathbb{Z}_m}

\newcommand{\intstwo}{\mathbb{Z} \times \mathbb{Z}}

\begin{document}

\title{When is a subgroup of a ring an ideal?}
\date{\today}

\author{Sunil K. Chebolu}
\address{Department of Mathematics \\
Illinois State University \\
Normal, IL 61790, USA} \email{schebol@ilstu.edu}

\author{Christina L. Henry}
\address{Department of Mathematics \\
Illinois State University \\
Normal, IL 61790, USA} \email{clhenry@ilstu.edu}

\thanks{
The first author is supported by an NSA grant (H98230-13-1-0238)}


\begin{abstract}
Let $R$ be a commutative ring.  When is a  subgroup of $(R, +)$  an ideal of $R$? We investigate this problem for the rings $\ints^{d}$ and $\prod_{i=1}^{d} \ints_{n_{i}}$. 
In the case of $\intstwo$ and $\intsnm$,  our results give, 
for any given subgroup of these rings, a computable criterion for  the problem under consideration. We also compute the probability that a randomly chosen subgroup from $\intsnm$ is an ideal.
\end{abstract}

\maketitle
\thispagestyle{empty}



\section{Introduction}
Let $R$ be a commutative ring.  
The object of this paper is to determine  necessary and sufficient conditions for a given subgroup of  $(R, +)$ to be an ideal of $R$.  Our motivation for asking this question arose from some problems on Mathieu subspaces (more is explained in the next paragraph).  To begin, consider the ring $\ints$ of integers. Every subgroup of $\ints$ is of the form $k \ints$ for some integer $k$, and each of these subgroups is clearly also an ideal. In fact, the same is true also for the rings $\ints_n$ (the ring of integers modulo $n$).  It turns out that these are the only rings $R$ in which every subgroup of $(R, +)$ is also an ideal of $R$; see Proposition \ref{zzn}.
In particular, when we consider product rings we get some subgroups that are not ideals.  For instance  the diagonal $\{(x, x) \, | \, x \in \ints\}$ in $\ints \times \ints$ is clearly a subgroup of $(\ints \times \ints, +)$ but not an ideal in the ring $\ints \times \ints$. In this paper we consider the product rings $\ints^{d}$ (in Section \ref{infinite}) and $\prod_{i=1}^{d} \ints_{n_{i}}$ (in Section \ref{finite}), and for various subgroups of these rings we give necessary and sufficient conditions for a given subgroup to be an ideal.   In the case of $\intstwo$ and $\intsnm$, our necessary and sufficient conditions are also computable for any given subgroup of these rings. As one would expect, our results show that in general  an arbitrary subgroup of a ring is seldom an ideal. In fact, we make this statement precise in Theorem \ref{counting} where we  compute explicitly the  probability that a randomly chosen subgroup from $\intsnm$ is an ideal. For instance, when $p$ is a prime and the ring is $\ints_p \times \ints_p$, this probability is only $\frac{4}{p+3}$.  We will  use several basic facts and tools from abstract algebra which can be found in \cite{DummitFoote}.  We also  use a theorem in group theory due to Goursat; a good exposition of this theorem can be found in \cite{jp}, and we review it in Theorem \ref{Goursat}. Although we focus mainly on  the rings $\intstwo$ and $\intsnm$, where possible we offer some generalizations. By a subgroup of a ring $R$, we always mean a subgroup of the additive group $(R, +)$.

This problem came up naturally when the first author and his collaborators (Yamskulna and Zhao) were recently working on some problems involving Mathieu subspaces in some rings. A Mathieu subspace is a generalization of a ideal: For a commutative ring $R$, a $\ints$-submodule $M$ of $R$ is said to be a Mathieu subspace of $R$ if whenever $a^n$ belongs to $M$  (for all $n \ge 1$), then $ra^n$ belongs to $M$ for all $n$ sufficiently large. Every ideal is a Mathieu subspace, but the converse is not necessarily true.  The notion of a Mathieu subspace was introduced by Wenhua Zhao in \cite{wz}, and it proved to be a central idea in the research on several landmark conjectures in algebra and geometry including the Jacobian conjecture. As a result, Mathieu subspaces received serious attention and extensive writing; see \cite{wz12} and references in it. 
Recently when the first author and his collaborators were working on some problems on Mathieu subspaces, they were led to the problem of determining when a subgroup of a ring is a Mathieu subspace. Since ideals are important and  relatively well-understood  classes of Mathieu subspaces, it was natural to investigate the same question for ideals. Thus the problem we study in this paper is an interesting offshoot of our Mathieu subspaces project. 

\vskip 3mm\noindent
\textbf{Acknowledgements:} We would like to thank  the referee for his/her  comments and suggestions which we used to improve the exposition of this paper.

\section{Generators} \label{generators}

In the introduction we noted that the rings $\ints$ and $\ints_n$ have the property that every subgroup in them  is also an ideal. It is not hard to show that these are the only rings with this property. 

\begin{prop} \label{zzn}
Let $R$ be a unital commutative ring. i.e., a commutative ring with a multiplicative identity.  If every subgroup of  $(R, +)$ is also an ideal, then $R$ is isomorphic to either $\ints$ or $\ints_n$ for some positive integer $n$.  
\end{prop}

\begin{proof}
Since $R$ is a unital ring, there is a natural map $\phi \colon \ints \lrar R$ which sends $1$ to $1_R$, the multiplicative identity of $R$. The image of this homomorphism is exactly the subgroup of $(R, +)$ that is generated by $1_R$. If every subgroup of $(R, +)$ is an ideal, then, in particular, the subgroup generated by  $1_R$ is also an ideal. However, the only ideal which contains $1_R$ is the entire ring $R$. This means $\phi$ is surjective. From the first isomorphism theorem, we have $\ints/\ker \phi \cong R$.   It follows that $R$ is isomorphic to $\ints$ or $\ints_n$ for some integer $n$. (In the former case $R$ has characteristic $0$, and in the latter $R$ has characteristic $n$.)
\end{proof}

We will now show that every subgroup of $\ints^{d}$ and $\prod_{i=1}^{d} \ints_{n_{i}}$  is generated by at most $d$ elements.   We will recall some standard results from abstract algebra which can be found in \cite{DummitFoote}. 

\begin{thm}
Let $R$ be a PID and let $M$ be a free $R$-module of rank $r$. Then every submodule of $M$ is also free and has rank at most $r$.
\end{thm}

This theorem takes care of $\ints^{d}$.  For $\prod_{i=1}^{d} \ints_{n_{i}}$, we need the following corollary which can be derived easily from the above theorem.

\begin{cor}
Let $R$ be a PID and let $M$ be a finitely generated $R$-module. If $M$ is generated by $r$ elements, then every submodule of $M$ is 
generated by at most $r$ elements.
\end{cor}

\begin{cor}
Every subgroup of $(\prod_{i=1}^{d} \ints_{n_{i}}, +)$ and that of $(\ints^{d}, +)$ is generated by at most $d$ elements.
\end{cor}

\begin{proof}
The ring $\prod_{i=1}^{d} \ints_{n_{i}}$ is a $\ints$-module that is clearly generated by $d$ elements; the standard basis forms a generating set.   Therefore by the above corollary 
every subgroup of $\prod_{i=1}^{d} \ints_{n_{i}}$ is generated by at most $d$ elements.  The corresponding statement for $\ints^{d}$ is a special case of the above theorem.
\end{proof}

This corollary gives a natural stratification of the class of  all non-subgroups of these rings which is based on the minimal number of generators of a given subgroup. This stratification will be helpful in our analysis.

\section{The ring $\ints \times \ints$} \label{infinite}
In this section we determine when a given additive subgroup of the ring  $\ints^{d}$  is  an ideal. The trivial subgroup which consists of the single element $(0, 0, \cdots, 0)$ is also trivially  an ideal, so we will consider non-zero subgroups. As explained in the previous section, a non-zero subgroup of $\ints^{d}$ is free of rank at most $d$.  We will be begin with rank 1 subgroups where the problem is straightforward.

\begin{prop}
Let $L$ be a subgroup of $\ints^{d}$ generated by $(a_{1}, \cdots, a_{d})$.  $L$ is an ideal if and only if all but one of the $a_{i}$'s are zero.
\end{prop}

\begin{proof}
If all but one of the $a_{i}$'s are zero, then $L$ is clearly an ideal in one of the factors of $\ints^{d}$.  On the other hand, if we have more than one non-zero $a_{i}$'s, say  $a_{i}$ and $a_{j}$,  then  consider $e_{i } = (0, \cdots 0, 1, 0 \cdots 0)$ which has one at the $ith$ spot. If $L$ is an ideal, then $e_{i}.(a_{1}, \cdots, a_{d})  = (0, \cdots 0, a_{i}, 0 \cdots 0)$ should belong to $L$. This is a contradiction, so we are done. 
\end{proof}

More generally, the following is true.

\begin{lemma}
Let $R$ be an integral domain. A subgroup of $(R, +)$ generated by a non-zero element $a$ is an ideal of $R$ if and only if $R$ is isomorphic to $\ints$  or $\ints_p$ for some prime $p$.
\end{lemma}

\begin{proof}
Let $\langle a \rangle$ be the additive subgroup of $(R, +)$ generated by $a ( \ne 0)$. Let $r$ be an arbitrary element of $R$.  If $\langle a \rangle$ is an ideal, then we should have $ra = na$ for some integer $n$. This equation implies that $ (r- n1_R)a = 0$. Since we are working in an integral domain and $a$ is non-zero, we get $r - n1_R = 0$, or $r = n1_R$. Since $r$ was arbitrary, this implies that $(R, +)$ is a cyclic group generated by $1_R$. This means $R$ is isomorphic to $\ints$ or $\ints_n$ for some $n$. But since $R$ is an integral domain, $n$ has to be a prime.
\end{proof}

Now we move on to subgroups of rank at least $2$ in $\ints^{d}$ where the problem is more interesting. We begin with an example to show the subtlety in the problem. 

\begin{example} Consider the ring $\intstwo$ and let $S$ and $T$ denote the following  rank two subgroups of $(\intstwo, +)$.
\begin{eqnarray*}
S & = & \langle (2, 0),  (3, 1) \rangle  \\
T & = &  \langle (2, 0) , (2, 1) \rangle 
\end{eqnarray*}
We claim that $S$ is not an ideal but $T$ is.  If $S$ is an ideal, then the element $(0, 1)$ ($= (0, 1) (3, 1)$) should belong to it.  That means the pair of equations $2x+ 3y = 0$ and $y = 1$ have to be consistent over $\ints$. However, it is easy to see that this is not the case.  On the other hand,  $T$ is an ideal in $\intstwo$. In fact,  $T = 2 \ints \times \ints$.  See Theorem \ref{2x2} for the general result.
\end{example}

 We begin by classifying ideals of $\ints^{d}$ whose additive groups are free of rank $k$.

\begin{prop} \label{prop:free}
Let $I$ be an ideal in $\ints^{d}$.  $I$ is free of rank $k$ $(1 \le k \le n)$ if and only if $I$ is of the form $\prod_{i=1}^{d} d_{i} \ints$ where exactly $k$ of the numbers $d_{i}$ are non-zero.
\end{prop}

\begin{proof}
Recall that every ideal in $\ints^{d}$ is of the form  $\prod_{i=1}^{d} d_{i} \ints$, where the $d_{i}$ are integers. The rank of $\prod_{i=1}^{d} d_{i} \ints$ is exactly the number of $d_{i}$s that are non-zero, so we are done. 
\end{proof}

In view of this proposition, to determine when a subgroup of rank $k$ in $\ints^{d}$ is an ideal, it is enough (after deleting the zero coordinates) to consider the problem  when $d=k$. The latter is addressed in the next two theorems.  We  begin with a lemma which we will need in these theorems. Recall that an integer matrix $A$ is said to be unimodular if it is invertible over the ring of integers. This statement is equivalent (as can be seen by Cramer's formula for the inverse) to saying that the determinant of $A$ is either $1$ or $-1$.
In the following lemma,  a subgroup of $\ints^{n}$ of rank $n$ will be called a lattice of  $\ints^{n}$.

\begin{lemma}
Let $A$ and $B$ be two $n \times n$ matrices over the integers that are invertible over the rationals. The columns of $A$ and those of $B$ form two bases for a  lattice $L$ if and only if there exists a unimodular matrix $X$ 
such that $AX= B$.  
\end{lemma}

\begin{proof}
Since the columns of $A$ and $B$ form a basis for $L$, there exist integer square matrices $X$ and 
$Y$  such that $AX = B$ and $B Y = A$.  Multiplying the first equation on the right hand side by $Y$, we get $AXY = BY$. But $BY = A$, so we get $AXY = A$. Since $A$ is invertible over the rationals, we multiply the inverse (over the rationals) of $A$ on both sides to conclude that $XY = I$. This means $X$ is invertible over $\ints$ (i.e, it is unimodular) and $AX = B$. For the other direction, let $Y$ be the inverse of $X$ over $\ints$, so  we have $AX = B$ and $BY = A$. The first equation tells us that the column space of $B$ is contained in that of $A$, and the second equation says that the  column space of $A$ is contained in that of $B$.
This completes the proof of the lemma.
\end{proof}

\begin{thm} \label{thmpart1}
Let $H$ be a subgroup of rank $k$ in $\ints^{k}$. Let the columns of a $k \times k $ matrix $A$ be a $\ints$-basis for  $H$. Then the following are equivalent. 
\begin{enumerate}
\item $H$ is an ideal in $\ints^{k}$ 
\item There exists a unimodular matrix $U$ such that $AU$ is a diagonal matrix.
\item There is a sequence of elementary row operations (over $\ints$) that can convert $A$ into a diagonal matrix
\end{enumerate}
\end{thm}

\begin{proof}
Let $H$ (as in the statement of the theorem) be an ideal in $\ints^{k}$. Then by Proposition \ref{prop:free}, $H$ is of the form $\prod_{i=1}^{k} d_{i} \ints$ for some integers $d_{i}$. Since $H$ has rank $k$, all these integers have to be non-zero. $H$ can be written in this form if and only if the columns of $A$ and those of the diagonal matrix $D = \text{Diagonal}(d_{1}, \cdots d_{k})$ form a basis for $H$. By the above lemma, this happens if and only if there is a unimodular matrix $U$ such that $AU = D$. Hence we have the equivalence of statements (1) and (2). The equivalence of (2) and (3) for the field of real numbers is well-known (the famous reduced row echelon form of an invertible matrix). The reader can verify that the proof works over $\ints$ when properly interpreted. For instance, the role played by non-zero real numbers in the world of $\ints$ are the units $\pm{1}$. That will give the equivalence of  statements (2) and (3).
\end{proof}

Since $\ints$ is a Euclidean domain where we can talk about gcds, we can take the above theorem one step further. Let $A^{*}$ denote the adjoint matrix of $A$. Recall that the formula for the inverse of $A$ (an invertible matrix) is given by $A^{-1} = \frac{1}{\det(A)} A^{*} = \frac{1}{\det(A)} ((a^{*}_{ij}))$.

\begin{thm} \label{thmpart2} Let $H$ be a subgroup of rank $k$ in $\ints^{k}$. Let the columns of a $k \times k $ matrix $A$ be a $\ints$-basis for  $H$. Then the following are equivalent. 
\begin{enumerate}
\item $H$ is an ideal in $\ints^{k}$ 
\item There exists a unimodular matrix $U$ such that $AU$ is a diagonal matrix.
\item There is a sequence of $k$ non-zero integers $d_{1}, d_{2}, \cdots d_{k}$ such that 
\begin{enumerate}
\item $\det(A) = \pm{d_{1}d_{2}\cdots d_{k}}$
\item $\det(A)/d_{i}$ divides $\gcd(a^{*}_{1i}, \cdots, a^{*}_{ki})$ for all $i$.
\end{enumerate}
\end{enumerate}
\end{thm}

\begin{proof}
We already saw the equivalence of (1) and (2) in Theorem \ref{thmpart1}.  Now we will show that (2) and (3) are equivalent.
Let $H$ and $A$ be  as in the statement of the theorem.  There exists a unimodular matrix $U$ such that $AU$ is a diagonal matrix if and only of for some diagonal matrix $D = \text{Diagonal}(d_{1}, \cdots d_{k})$, $A^{-1}D$ is unimodular. Using Cramer's formula for the inverse, we can equivalently say that 
\[ X = \frac{1}{\det(A)} A^{*} D\]
is unimodular.  Since $X$ is unimodular, its determinant is $\pm{1}$. Taking determinants of both sides of the above matrix equation will give (a). Moreover, the entries of $X$ should be all integers. For that to happen, $\det(A)$ should divide all the entries in each of the columns $d_{i}(a^{*}_{1i}, \cdots, a^{*}_{ki})^{T}$, or equivalently $\det(A)/d_{i}$ should divide all the entries in each of the columns $(a^{*}_{1i}, \cdots, a^{*}_{ki})^{T}$.  Since $\ints$ is a Euclidean domain, the last statement is equivalent to (b).
\end{proof}

We can tell exactly when the condition (2) of Theorem \ref{thmpart2} holds in the case of $\intstwo$.  That gives the following result, which along with the rank $1$ result proved earlier gives a full answer to our problem for the ring $\intstwo$.

\begin{thm} \label{2x2}
Let $L$ be a rank $2$ subgroup of $\intstwo$ that is generated by vectors $(a, b)$ and $(c, d)$. $L$ is an ideal in $\intstwo$ if and only if $ad-bc$ divides $\gcd(a, c) \cdot \gcd(b, d)$.
\end{thm}

\begin{proof}
Let $L$ be a rank $2$ subgroup of  $\intstwo$  that is generated by vectors $(a, b)$ and $(c, d)$, and let $A$ be the $2 \times 2$ matrix with these two columns. From the above theorems, and using the formula for the inverse of a $2 \times 2$ matrix,  we conclude that $L$ is an ideal if and only if there exists non-zero integers $d_{1}$ and $d_{2}$ such that 
\begin{enumerate}
\item $ad-bc= \pm{d_{1}d_{2}}$
\item $(ad-bc)/d_{1}$ divides $\gcd(b, d)$  and $(ad-bc)/d_{2}$ divides $\gcd(a, c)$.
\end{enumerate}
We claim that non-zero integers $d_{1}$ and $d_{2}$ exist with these properties if and only if $ad-bc$ divides $\gcd(a, c) \cdot \gcd(b, d)$. If $d_{1}$ and $d_{2}$ exist such that (1) and (2) hold, then from (2) we get $(ad-bc)^{2}/ (d_{1}d_{2})$ divides $\gcd(a, c) \cdot \gcd(b, d)$, but $(ad-bc)^{2}/ (d_{1}d_{2}) = ad-bc$. This proves one direction. For the other direction, suppose $ad-bc$ divides $\gcd(a, c) \cdot \gcd(b, d)$.  Then an elementary number theory fact tells us we can write $ad-bc$ as $d_{1}d_{2}$ where $d_{1}$ divides $\gcd(a, c)$ and $d_{2}$ divides $\gcd(b, d)$. 
\end{proof}

We now explain how one can arrive at Theorem \ref{2x2} more directly by solving linear equations over $\ints$. Recall that our problem boils down to the following question. \emph{Given an integer matrix $A$ with non-zero determinant, when does there exist a unimodular matrix $X$ such that $AX$ is a diagonal matrix?} To address this, we let $X = (x_{ij})$ and consider the matrix equation 
\[
\begin{bmatrix}  a & c \\ b &  d\end{bmatrix}   \begin{bmatrix}  x_{11} & x_{12} \\ x_{21} &  x_{22}\end{bmatrix}=  \begin{bmatrix}  u & 0 \\ 0 & v \end{bmatrix}
.\]
This gives us the following set of equations:
\begin{eqnarray}
a x_{12} + c x_{22} & = & 0 \\
b x_{11} + d x_{21} & = & 0 \\ 
x_{11}x_{22} - x_{12}x_{21} &=& 1
\end{eqnarray}
($X$ is unimodular, so its determinant  
is either $1$ or $-1$.  However, by swapping the columns of $A$ if necessary, we may assume that the determinant of $X$ is $1$ which gives us the third equation.)
$L$ is an ideal if and only if the above system of equations has a solution in integers $x_{ij}$.  Let us begin with equation 1: 
$a x_{12} + c x_{22} =  0$ if and only if $a x_{12}  = - c x_{22}$. Then
\[ x_{12} = \frac{-c}{\gcd(a, c)} \alpha  \ \ \text{ and }  \ \ x_{22} = \frac{a}{\gcd(a, c)} \alpha \ \ \text{for some integer } \alpha.\]
Similarly, using equation 2, we get 
\[ x_{11} = \frac{-d}{\gcd(b, d)} \beta  \ \ \text{ and }  \ \ x_{21} = \frac{b}{\gcd(b, d)} \beta,  \ \ \text{for some integer } \beta.\]
Substituting these values in the determinant condition (equation 3), we get 
\begin{eqnarray*}
x_{11}x_{22} - x_{12}x_{21} &=& 1\\
\frac{-d}{\gcd(b, d)} \beta  \frac{a}{\gcd(a, c)}  \alpha -  \frac{-c}{\gcd(a, c)} \alpha \frac{b}{\gcd(b, d)} \beta & = &  1\\
\alpha \beta \left(\frac{-ad}{ \gcd(a, c)\gcd(b, d)}   - \frac{-bc}{\gcd(a, c) \gcd(b, d)} \right) & = &  1 \\
- \alpha \beta (ad - bc) & = &  \gcd(a, c) \gcd(b, d) 
\end{eqnarray*}
Thus we see from the last equation that the above system of equations is consistent over $\ints$ if and only if  $\det(A) = ad-bc$ divides $\gcd(a, c) \gcd(b, d)$ in $\ints$. (In that case, we can take $\alpha = -1$ and $\beta = \frac{\gcd(a, c) \gcd(b, d) }{ad-bc}.)$ This completes the alternative proof of Theorem \ref{2x2}.

The following corollary follows immediately from Theorem \ref{2x2}.

\begin{cor}
Let  $(a, b)$ and $(c, d)$ be two vectors in $\intstwo$ and $L$ be the lattice generated by these two vectors. 
\begin{enumerate}
\item  If $ad-bc= \pm 1$, then $L$  is an ideal in $\intstwo$.
\item If $ad-bc$ is a prime, then $L$ is an ideal if and only if $ad-bc$ divides either $\gcd(a, c)$ or $\gcd(b, d)$.
\end{enumerate} 
 \end{cor}

\section{The ring $\intsnm$} \label{finite}
Let $n$ and $m$ be positive integers and consider the ring $\intsnm$. Our problem is to determine when a subgroup of $(\intsnm, +)$ is an ideal.  We have seen that a non-zero subgroup of $\intsnm$ is generated by either one or two elements, so we have two cases to consider. First, consider a subgroup $L$ in the ring $\intsnm$  that is generated by $(a, b)$.   If either $a  = 0$ in $\intsn$ or $b = 0$ in $\mathbb{Z}_m$, the problem is trivial because $L$ is simply an ideal in one of the components of $\intsnm$. So let us assume that  both $a$ and $b$ are non-zero in their respective component rings. Then we have the following theorem.

\begin{thm}  \label{gcd}
Let $1 \le a < n$ and $1 \le b < m$. The subgroup generated by  $(a, b)$ in the ring $\intsnm$ is a ideal if and only if 
\[ \gcd\left( \frac{n}{\gcd(a, n)}, \frac{m}{\gcd(b,m)}\right) = 1.\]
\end{thm}

\begin{proof}
Since our rings are principal ideal rings,  every ideal in $\intsnm$ is of the form $d_1\intsn \times d_2\mathbb{Z}_m$, where $d_1$ and $d_2$ are some integers. For brevity we will denote this ideal by $ \langle d_1 \rangle \times  \langle d_2 \rangle $. 

Returning to our problem, let us assume that the line $L$ generated by $(a, b)$ is an ideal of $\intsnm$. From above, we have
\[ L =\langle d_1 \rangle \times \langle d_2 \rangle. \]
  Consider the restrictions to $L$ of the natural projection maps: $\pi_1 \colon \intsnm \rar  \intsn$ and $\pi_2 \colon \intsnm \rar  \mathbb{Z}_m$. 
We will compute $\pi_1(L)$ in two different ways. On the one hand, since $L = \langle d_1 \rangle \times \langle d_2 \rangle$, we have $\pi_1(L) = \langle d_1\rangle$. On the other hand, $L$ is generated by $(a,b)$, so the first components of the elements of $L$ pick up all multiples of $a$. Therefore $\pi_1(L) = \langle a \rangle$. This shows that $\langle a \rangle= \langle d_1 \rangle $. Similarly, working with the second projection map, we conclude that $\langle b \rangle  = \langle d_2 \rangle$.

To summarize, $L$ spanned by $(a, b)$ is an ideal if and only if 
\[ \langle (a, b) \rangle  = \langle a \rangle  \times \langle b \rangle.\] 
The inclusion $ \langle (a, b) \rangle  \subseteq \langle a \rangle  \times \langle b \rangle$ is obvious. Therefore, equality holds if and only if both sides have the same cardinality. These cardinalities are given by the following formulas ($\ord(x)$ denotes the additive order of $x$).
\begin{eqnarray*}
|\langle (a, b) \rangle | & = &  \text{lcm}(\ord(a), \ord(b)) = \frac{\ord(a) \, \ord(b)}{\gcd(\ord(a), \ord(b)) }\\
| \langle a \rangle \times \langle  b \rangle |   &   = & \ord(a) \, \ord(b) 
\end{eqnarray*}
Equating these two expressions, clearly $L$ spanned by $(a, b)$ in $\intsnm$ is ideal if and only $\gcd(\ord(a), \ord(b)) = 1$.
The theorem now follows from the fact that the order of  an element $c$ in $(\mathbb{Z}_s, +)$ is given by $ \frac{s}{\gcd(c, s)}$.
\end{proof}

\begin{rem}
When $m$ and $n$ are relatively prime, Theorem \ref{gcd} implies that every line in $\intsnm$ is an ideal. This is indeed the case because for relatively prime integers $m$ and $n$ we have $\ints_n \times \ints_m \cong \ints_{nm}$.
\end{rem}

More generally, the following theorem is true:

\begin{thm}
The subgroup generated by the element $(a_{1}, a_{2}, \cdots, a_{k})$ in $\ints_{n_{1}} \times  \ints_{n_{2}} \times \cdots \times \ints_{n_{k}}$ is an ideal if and only if 
\[ \prod_{1 \le i < j \le n} \gcd\left( \frac{n_{i}}{\gcd(a_{i}, n_{i})},  \frac{n_{j}}{\gcd(a_{j}, n_{j})} \right) = 1. \]
\end{thm}

\begin{proof}
From the proof of Theorem \ref{gcd}, it follows that the subgroup generated by the element $(a_{1}, a_{2}, \cdots, a_{k})$ in $\ints_{n_{1}} \times  \ints_{n_{2}} \times \cdots \times \ints_{n_{k}}$ is an ideal if and only if 
\[ \prod_{i} \ord(a_{i})  = \underset{i}{\text{lcm}} \; \ord(a_{i}).\]
Showing that this last equation holds if and only if 
\[  \prod_{1 \le i < j \le n}  \gcd(\ord(a_{i}), \ord(a_{j})) = 1 \]
can be done as an exercise.  Then using the formula mentioned above for the order of an element in $\ints_{s}$, we now  get the condition given in the statement of the theorem.
\end{proof}

We now investigate when a subgroup of $\intsnm$ generated by two elements is an ideal.  To this end, the following theorem from group theory due to Goursat will be useful. We will also use this theorem in the next section where we compute some probabilities.

\begin{thm} (Goursat) \cite{jp} \label{Goursat}
Let $G_1$ and $G_2$ be any two groups.  There exists a bijection between the set $S$ of all subgroups of $G_1 \times G_2$ and the set $T$ of all $5$-tuples $(A_1, B_1, A_2, B_2, \phi)$ where $A_i$ is a subgroup of $G_i$, $B_i$ is a normal subgroup of $A_i$, and $\phi$ is a group isomorphism from $A_1/B_1$ to $A_2/B_2$. 
 \end{thm}
 
 Let $\pi_i \colon G_1 \times G_2 \rar G_i$ denote the projection homomorphisms.  The desired bijection in this theorem is given as follows. For a subgroup $U$ of $G_1 \times G_2$, we define a $5$-tuple $(A_{U_1}, B_{U_1}, A_{U_2}, B_{U_2}, \phi_U)$ where 
 \begin{eqnarray*}
 A_{U_1} & = &  \text{Im}(\pi_1|_U)\\
 B_{U_1} & = & \pi_1(\ker(\pi_2|_U))\\
  A_{U_2} & = &  \text{Im}(\pi_2|_U)\\
 B_{U_2} & = & \pi_2(\ker(\pi_1|_U)) \text{ and }\\
 \phi_U(a_1B_{U_1}) & = & a_2B_{U_2} \text{ when } (a_1, a_2) \in U. 
 \end{eqnarray*}

Conversely, given  a $5$-tuple $(A_1, B_1, A_2, B_2, \phi)$, the corresponding subgroup $U$ of $G_1 \times G_2$ is given by 
\[U_{\phi} = \{ (a_1, a_2) \in A_1 \times A_2 \, | \, \phi(a_1 B_1) = a_2 B_2\}.\]

\begin{cor} \label{cor:Goursat}
Let $G_1 \times G_2$ be a finite group and let $(A_{U_1}, B_{U_1}, A_{U_2}, B_{U_2}, \phi_U)$  correspond to the subgroup $U$ of  $G_1 \times G_2$. Then we have  
\[ | U | = |A_{U_1}||B_{U_2}|. \] 
\end{cor}

\begin{proof}
It is clear from the correspondence in Goursat's theorem that 
\[|U| = |A_{U_1}/B_{U_1}| |B_{U_1}||B_{U_2}| = |A_{U_1}||B_{U_2}|.\]
\end{proof}

Given elements $\alpha$ and $\beta$ in $\ints_n$, consider the linear map $\phi_{\alpha, \beta} \colon \ints \times \ints \rar \ints_n$ defined by $\phi_{\alpha, \beta}(x, y) = \alpha x + \beta y$. Then we have the following theorem.

\begin{thm} \label{thm:Goursat}
The subgroup of $\intsnm$ generated by $(a, b)$ and $(c, d)$ is an ideal  of $\intsnm$ if and only if 
\[(\ker \phi_{a, c})(\ker \phi_{b, d}) = \intstwo. \]
\end{thm}

\begin{proof}
Let $H$ denote the subgroup generated by $(a, b)$ and $(c, d)$ in $\intsnm$. Suppose $H$ is an ideal in $\intsnm$. Then there exists $\alpha$ in $\ints_n$ and $\beta$  in $\ints_m$ such that $H = \langle \alpha \rangle  \times  \langle \beta \rangle$. Taking projection maps, we can see that $\alpha = \gcd(a, c) \mod n$ and $\beta = \gcd(b, d) \mod m$. Thus $H$ is an ideal if and only if $\langle (a, b), (c, d) \rangle = \langle \gcd(a, c) \rangle \times \langle \gcd(b, d) \rangle$. As in Theorem 4.1, the left hand side is easily seen to be contained in the right hand side and we have equality if and only if both sides have the same cardinality.
The cardinality of the right hand side is $\ord( \gcd(a, c)) \ord( \gcd(b, d))$.  The cardinality of the left hand side can be computed using Corollary \ref{cor:Goursat}: It is given by $\ord( \gcd(a, c)) |\pi_2(\ker \pi_1|_H)|$. Equating these two expressions, we conclude  that $H$ is an ideal if and only if $\ord( \gcd(b, d)) = |\pi_2(\ker \pi_1|_H)|$. The left hand side of this equation is the cardinality of the set 
\[ S = \{ bx + dy \, | \, x, y \in \ints \} \subseteq \ints_{m},\]
and the right hand side is the cardinality of the set 
\[ T = \{ bx + dy\, | \, x, y \in \ints \text{ such that } ax+cy = 0 \in \ints_n\} \subseteq \ints_{m}.\]
$S$ and $T$ have the same cardinality precisely when the image of $\phi_{b,d} \colon \intstwo \rar \ints_m$ is the same as the image of $\phi_{b, d}$ restricted to the kernel of $\phi_{a, c} \colon \intstwo \rar \ints_n$. That happens exactly when $\ker(\phi_{a,c})$ intersects every coset in $\intstwo/\ker(\phi_{b,d})$ which is true if and only if $(\ker \phi_{a, c})(\ker \phi_{b, d}) = \intstwo.$
\end{proof}

We can get a finite-type condition  that is equivalent to the one given in Theorem \ref{thm:Goursat}. To get this, set $l = \text{lcm}(m, n)$. Then given elements $\alpha$ and $\beta$ in $\ints_n$, define the linear map $\psi_{\alpha, \beta} \colon \ints_{l} \times \ints_{l} \rar \ints_n$ as $\psi_{\alpha, \beta}(x, y) = \alpha x + \beta y$.  We now have the following corollary.

\begin{cor}
The subgroup of $\intsnm$ generated by $(a, b)$ and $(c, d)$ is an ideal  of $\intsnm$ if and only if 
\[|(\ker \psi_{a, c})(\ker \psi_{b, d})| = nm. \]
\end{cor}

\begin{proof}
This follows from the proof of the previous theorem. Note that the maps $\phi_{a,c}$ and $\phi_{b, d}$ factor through $\psi_{a,c}$ and $\psi_{b, d}$ respectively. 
\end{proof}

Goursat's theorem for more than two components \cite{genGoursat} has a very complicated structure and in particular, it is not helpful to solve our problem.

\section{Probability for a subgroup to be an ideal}
As one would expect, the above results suggest that a subgroup of a ring is rarely an ideal.  Now we will make this precise by computing explicitly the probability that a randomly chosen subgroup of $\intsnm$ is an ideal using the approach and results from \cite{jp}. Let $P_R$ denote the probability that a randomly chosen subgroup of a finite ring $R$ is an ideal. This probability is given by 
\[ P_R = \frac{\text{ total number of ideals in  } R} { \text{total number of subgroups in }(R, +)}.\]

Our interest is in the ring $\intsnm$. If either $n$ or $m$ is one, then clearly $P_R =1$. So we will assume that $n > 1$ and $m >1$.  Let $S = \{ p_1, \cdots p_k\}$ denote the set of all distinct primes which divide $mn$.  Then the prime factorizations of $m$ and $n$ are given by
\[ m = p_1^{r_1} \cdots p_k^{r_k} \text { and } n = p_1^{s_1} \cdots p_k^{s_k},\]
where the exponents are non-negative integers, and the Chinese remainder theorem gives the decomposition
\[\intsnm = \left( \ints_{p_1^{r_1}} \times \ints_{p_1^{s_1}} \right)  \times \cdots   \times \left( \ints_{p_k^{r_k}} \times \ints_{p_k^{s_k}} \right). \]

\begin{lemma} \label{reduction}
\[P_{\intsnm} = \prod_{i=1}^k P_{\ints_{p_1^{r_1}} \times \ints_{p_1^{s_1}}}\]
\end{lemma}

\begin{proof}
This follows from two facts. First, note that every ideal $I$ in $\intsnm$ is of the form $I  = \prod_{i=1}^k I_i$ where $I_i$ is an ideal of $\ints_{p_i^{r_i}} \times \ints_{p_i^{s_i}}$. 
Next we use a theorem of Suzuki \cite{suz} which says if $G_{1}$ and $G_{2}$ are two finite groups with relatively prime orders, then every subgroup of $G_{1} \times G_{2}$ is of the form $H_{1} \times H_{2}$, where $H_{i}$ is a subgroup of $G_{i}$.  In particular, every subgroup $H$ of $(\intsnm, +)$ is of the form $\prod_{i=1}^k H_i$ where $H_i$ is a subgroup of $\ints_{p_i^{r_i}} \times \ints_{p_i^{s_i}}$.  Then we have the following equations which complete the proof of the lemma.
\begin{eqnarray*}
P_{\intsnm}  & = &  \frac{\text{ total number of ideals in  } \intsnm} { \text{total number of subgroups in }(\intsnm, +)} \\
                    & = &  \prod_{i=1}^k  \frac{\text{ total number of ideals in  } \ints_{p_i^{r_i}} \times \ints_{p_i^{s_i}}} { \text{total number of subgroups in }(\ints_{p_i^{r_i}} \times \ints_{p_i^{s_i}}, +)}\\
                    & = &  \prod_{i=1}^k P_{\ints_{p_i^{r_i}} \times \ints_{p_i^{s_i}}}.
\end{eqnarray*}
\end{proof}

In view of this lemma, it is enough to compute $P_{\ints_{p_i^{r_i}} \times \ints_{p_i^{s_i}}}$.  We will do this in the next two lemmas, beginning by computing the number of ideals.

\begin{lemma}
The number of ideals in $\ints_{p^{r}} \times \ints_{p^{s}}$ is equal to $(r+1)(s+1)$.
\end{lemma}

\begin{proof}
Every ideal in $\ints_{p^{r}} \times \ints_{p^{s}}$ is of the form $ a\ints_{p^{r}} \times b \ints_{p^{s}}$, where $a$ is a divisor of $p^{r}$ and $b$ is a divisor of $p^{s}$. This gives $(r+1)(s+1)$ for the total number of ideals.
\end{proof}

Next we have to compute the number of subgroups in $\ints_{p^{r}} \times \ints_{p^{s}}$. This number can be obtained using the above-mentioned Goursat's theorem and can be found in \cite{jp}.

\begin{lemma} \cite{jp}
The total number of subgroups of $\ints_{p^{r}} \times \ints_{p^{s}}$  ($r \le s $) is given by 
\[\frac{p^{r+1}[(s-r+1)(p-1)+2]- [(s+r+3)(p-1)+2]}{(p-1)^2}\]
\end{lemma}

\noindent
\emph{Proof Sketch:}  Goursat's theorem can be greatly simplified in the case under consideration. There is a unique subgroup of order $p^k$ in $\ints_{p^{r}}$ for any $0 \le k \le r$ and these subgroups form a linear chain.  Moreover the group of automorphisms of $\ints_{p^k}$ corresponds to the units in this ring, and we have $p^k - p^{k-1}$ of them.  We now have to count the 5-tuples $(A_{1}, B_{1}, A_{2}, B_{2}, \phi)$ which correspond to subgroups in Goursat's theorem. If $|A_i/B_i| = 1$, the number of subgroups is $(r+1)(s+1)$ because we have $r+1$ choices for $A_1/B_1$ and $s+1$ choices for $A_2/B_2$ (clearly $\phi$ is trivial).
If $|A_i/B_i|  = p^k$ for $1 \le k \le r$, we have $r-k+1$ choices for $A_1/B_1$ and $s-k+1$ choices for $A_2/B_2$, and finally $p^k - p^{k-1}$ choices for $\phi$, so in this case we have $(r-k+1)(s-k+1)(p^k-p^{k-1})$ subgroups. In total we have
\[ (r+1)(s+1) + \sum_{k=1}^r (r-k+1)(s-k+1) (p^k-p^{k-1})\]
subgroups. The rest is straightforward algebra; see \cite{jp}. 

\vskip 3mm
Combining the above lemmas, we get our formulas for $P_{\ints_{p^r} \times \ints{p^s}}$ and $P_{\intsnm}$.

\begin{thm} \label{counting}
Let $p$ be a prime and let $r, s \, (r \le s), n$ and $m$ be  positive integers.
\[P_{\ints_{p^{r}} \times \ints_{p^{s}}} = \frac{(r+1)(s+1)(p-1)^2}{p^{r+1}[(s-r+1)(p-1)+2]- [(s+r+3)(p-1)+2]} \]
\[ P_{\intsnm} = \prod_{i=1}^k \frac{(r_i+1)(s_i+1)(p_i-1)^2}{p_i^{r_i+1}[(|s_i-r_i|+1)(p_i-1)+2]- [(s_i+r_i+3)(p_i-1)+2]}\]
\end{thm}

We now record two special cases which can be be derived from Theorem \ref{counting} using routine algebra. 

\begin{cor} Let $p$ be a prime and let $r$ be a positive integer.
\[P_{\ints_{p^{r}} \times \ints_{p^{r}}} = \frac{(r+1)^2(p-1)^2}{p^{r+1}(p+1) -2r(p-1)-3p+1} \]
\[P_{\ints_p \times \ints_p} = \frac{4}{p+3}\]
\end{cor}

It is clear from the above expressions that these probabilities are small, as expected. For instance, by choosing a large prime the value of $P_{\ints_p \times \ints_p}$ can be made arbitrarily small. Similarly for a fixed prime $p$, the numerator of $P_{\ints_{p^{r}} \times \ints_{p^{r}}}$ is a polynomial function in $r$ whereas the denominator is an exponential function in $r$.  Thus $\lim_{r \rar \infty} P_{\ints_{p^{r}} \times \ints_{p^{r}}} = 0$.

The main obstruction in generalizing these formulas to the rings $R = \prod_{i=1}^k \ints_{n_i}$ is the lack of a closed formula for the number of subgroups in $(\prod_{i=1}^k \ints_{p^i}, +)$ when $k \ge 3$. However, when the integers $n_i$ are all square-free, one can compute $P_R$ easily. This is because Lemma \ref{reduction} helps us to reduce the problem of  computing $P_R$ to the problem of   computing  $P_{S}$ where $S = \prod_{i=1}^r \ints_p$ for some prime $p$ and positive integer $r$ ($\le k$). The latter is a vector space over $\mathbb{F}_p$ where subgroups are same as vector subspaces. The number of subspaces in $(S, +)$ is given by the well-known formula
\[ \sum_{i=1}^{r}  {r \choose  i}_{p}\]
where ${r \choose  i}_{p}$ is the Gaussian binomial coefficient which counts the number of $i$-dimensional subspaces of 
$\mathbb{F}_{p}^{r}$. Explicitly its  value is given by 
 \[{r \choose  i}_{p} = \frac{(p^{r}-1)(p^{r}-p)\cdots(p^{r}-p^{r-1})}{(p^{i}-1)(p^{i}-p)\cdots(p^{i}-p^{i-1})}.\]
Since the number of ideals in $S$ is $2^r$, we get
\begin{prop} 
\[P_{\ints_{p}^{r}} = \frac{2^r}{\sum_{i=1}^{r}  {r \choose  i}_{p}}.\] 
\end{prop}

\end{document}